\documentclass[12pt]{amsart}
\usepackage{enumerate}
\usepackage{amsthm}
\usepackage{amsmath,amssymb,latexsym,amscd}
\usepackage{mathrsfs}
\usepackage{tikz}
\makeatletter
\@namedef{subjclassname@2010}{%
  \textup{2010} Mathematics Subject Classification}
\makeatother
\theoremstyle{definition}
\newtheorem{theorem}{Theorem}[section]

\newtheorem{proposition}[theorem]{Proposition}

\theoremstyle{definition}
\theoremstyle{definition}

\newtheorem{example}[theorem]{Example}
\usepackage{indentfirst}

\def\Z{\mathbb{Z}}

\def\S{\mathbb{S}}

\begin{document}
\baselineskip=17pt
\title[]{Cohomology classification of spaces with free $\S^3$-Actions }
\author[Anju Kumari and Hemant Kumar Singh]{ Anju Kumari and Hemant Kumar Singh}
\address{ Anju Kumari\newline 
	\indent Department of Mathematics\indent \newline\indent University of Delhi\newline\indent 
	Delhi -- 110007, India.}
\email{anjukumari0702@gmail.com}
\address{  Hemant Kumar Singh\newline\indent 
Department of Mathematics\newline\indent University of Delhi\newline\indent 
Delhi -- 110007, India.}
\email{hemantksingh@maths.du.ac.in}

\date{}
\thanks{This paper is supported by the Science and Engineering Research Board (Department of Science and Technology, Government of India) with reference number- EMR/2017/002192}
\begin{abstract} 
	
This paper gives the cohomology classification of finitistic spaces $X$ equipped with free actions of the group $G=\mathbb{S}^3$ and the orbit space $X/G$ is the integral or  mod 2 cohomology quaternion projective space $\mathbb{HP}^n$. We have proved that $X$ is the integral or  mod 2 cohomology  $\mathbb{S}^{4n+3}$ or $\mathbb{S}^3\times \mathbb{HP}^n$. Similar results for $G=\mathbb{S}^1$ actions are also discussed.
\end{abstract}


\maketitle
\section {Introduction}

Let $G$ be a compact Lie group acting on a finitistic space $X$. There are interesting problems related to transformation groups, for example, to classify the fixed point set $X^G$, the existence of free/semifree actions and the study of the orbit space $X/G$ for free actions of $G$ on $X$.  A  number of results has been proved in the literature in this direction  \cite{Baszczyk,Pinka Dey,Dotzel,Harvey,Pergher,Mahender}. An another thread of research is to classify $X$ for a given orbit space $X/G$ when $G$ acts freely on $X$. 
 Su\cite{Su1963} proved that if $G=\mathbb{S}^d$, $d=0,1$, acts freely on a space $X$ and the orbit space $X/G$   is cohomology $\mathbb{FP}^n$, then  space $X$ is the cohomology sphere $\mathbb{S}^{(d+1)n+d}$, when $d=0$, $\mathbb{F}=\mathbb{R}$ with $\mathbb{Z}_2$ coefficients, and when $d=1$, $\mathbb{F}=\mathbb{C}$ with integer coefficients.  He also proved that if $G=\mathbb{Z}_p$, $p$ an odd prime, acting freely on a space $X$ with the orbit space the mod $p$ cohomology Lens space $\text{L}_p^{2n+1}$,  then $X$ is the mod $p$ cohomology $(2n+1)$-sphere $\mathbb{S}^{2n+1}$.
  Kaur et al. \cite{JKaur2015} shown that if $G=\mathbb{S}^3$ acts freely on the mod 2 cohomology $n$-sphere $\mathbb{S}^n$, then $n\equiv3\text{(mod 4)}$ and the orbit space is the mod 2 cohomology quaternion  projective space $\mathbb{HP}^n$. In this paper, we have  shown that if $G=\mathbb{S}^3$  acts freely on a finitistic space $X$ with the orbit space the mod 2 cohomology quaternion projective space, then $X$ is the mod 2 cohomology $\mathbb{S}^{4n+3}$ or $\mathbb{S}^3\times \mathbb{HP}^n$ depending upon the Euler class of the associated bundle is nontrivial or trivial. A similar result with the integer coefficient  is also discussed. 
We have also proved  Kaur's results \cite{JKaur2015} with integer coefficients.

For the actions of  $G=\mathbb{S}^1$,  Su\cite{Su1963} proved that if $G=\mathbb{S}^1$ acts freely on a space $X$ such that $X/G$ is a cohomology complex projective space with $\text{dim}_{\mathbb{Z}}X/G<\infty$ and $\pi^*:H^2(X/G)\rightarrow H^2(X)$, where $\pi:X\rightarrow X/G$ is the orbit map, is trivial, then $X$ is an integral cohomology $(2n+1)$-sphere.
 We  have discussed the case when the induced map $\pi^*$ is nontrivial. In this case, we have proved that $X$ is the integral cohomology $\mathbb{S}^1\times\mathbb{CP}^n$. With coefficients in $\mathbb{Z}_p$, $p$ a prime, we have also shown that $X$ is the mod $p$ cohomology $\mathbb{S}^{2n+1}$ or $\mathbb{S}^1\times\mathbb{CP}^n$ or $L_p^{2n+1}$.

\section{Preliminaries}
Let $G$ be a compact Lie group and $G\to E_G\to B_G$ be the universal principal $G$-bundle, where $B_G$ is the classifying space. Suppose $G$  acts freely on a  space $X$. The associated bundle $X\hookrightarrow (X\times E_G)/{G}\to B_G$ is a fibre bundle with fibre $X$. Put $X_G=(X\times E_G)/{G}$. Then the bundle $X\hookrightarrow X_G\to B_G$ is called the Borel fibration. We consider the Leray-Serre spectral sequence for the Borel fibration. If $B_G$ is simply connected, then the system of local coefficients on $B_G$ is simple and the $E_2$-term of the Leray-Serre spectral sequence corresponding to the Borel fibration becomes
\begin{equation*}
E_2^{k,l}= H^k(B_G;R)\otimes H^l(X;R).
\end{equation*}
For details about spectral sequences, we refer \cite{McCleary}.
Let $h:X_G\rightarrow X/G$ be the map induced by the $G$-equivariant projection $X\times E_G\rightarrow X$.  Then, h is a homotopy equivalence \cite{Dieck}.

The following results  are needed to prove our results:

\begin{proposition}[\cite{Hatcher}]
	Let $R$ denote a ring and $\S^{n-1}\to E\stackrel{\pi}{\rightarrow} B$ be an oriented sphere bundle. The following sequence is exact  with  coefficients in $R$
	\begin{align*}
	\cdots\rightarrow H^{i}(E)\stackrel{\rho}{\rightarrow}H^{i-n+1}(B)\stackrel{\cup}{\rightarrow}H^{i+1}(B) \stackrel{\pi^*}{\rightarrow}H^{i+1}(E)\stackrel{\rho}{\rightarrow}H^{i-n+2}(B)\rightarrow\cdots
	\end{align*}
	which  start with 
	\begin{align*}
	0\rightarrow &H^{n-1}(B)\stackrel{\pi^*}{\rightarrow}H^{n-1}(E)\stackrel{\rho}{\rightarrow}H^0(B)\stackrel{\cup}{\rightarrow}H^n(B)\stackrel{\pi^*}{\rightarrow}H^n(E)\rightarrow\cdots 	\end{align*}
	where $\cup:H^i(B)\to H^{i+n}(B)$  maps $x\to x\cup u$ and $u\in H^n(B)$ denotes the Euler class of the  sphere  bundle. The above exact sequence is called the Gysin sequence. It is easy to observe that $\pi^*:H^i(E)\to H^i(B)$ is an isomorphism for all $0\leq i< n-1$.
\end{proposition}
 \begin{proposition}\cite{JKaur2015}\label{H^j(X/G)=0}
	Let $A$ be an $R$-module, where $R$ is PID, and $G=\mathbb{S}^3$ acts freely on a finitistic space $X$. Suppose that $H^j(X,A)=0$ for all $j>n$, then $H^j(X/G,A)=0$ for all $j>n$.
\end{proposition}
We have taken \v{C}ech cohomology  and all spaces are assumed to be finitistic. Note that $X\sim_R Y$ means $H^*(X;R)\cong H^*(Y;R)$, where $R=\mathbb{Z}_2$ or $\mathbb{Z}$.

\section{Main Theorems}

Recall that the projective spaces $\mathbb{FP}^n$ are the orbit spaces  of standard free actions of $G=\mathbb{S}^d$ on $\mathbb{S}^{(d+1)n+d}$, where $\mathbb{F}=\mathbb{C}$ or $\mathbb{H}$ for $d=1$ or $3$, respectively. If we take a free action of $\mathbb{S}^d$  on itself and the trivial action on $\mathbb{FP}^n$, then the orbit space of this diagonal action is $\mathbb{FP}^n.$ Now, the natural question: Is the converse true? If $G$ acts freely on a finitistic space $X$ with $X/G\sim_R \mathbb{FP}^n$, then 
whether $X\sim_R\mathbb{S}^{(d+1)n+d}$
 or  $X\sim_R \mathbb{S}^d\times\mathbb{FP}^n$.
  In the following theorems, we have proved that the converse of these statements are  true.

\begin{theorem}\label{Theorem 1}
	Let $G=\S^3$ acts freely on a finitistic space $X$ with $X/G\sim_R \mathbb{HP}^n$, where $R=\mathbb{Z}_2$ or $\mathbb{Z}$, and   $u\in H^4(X/G)$ be the Euler class of  the bundle $G\hookrightarrow X\stackrel{\pi}{\to} X/G$. Then, $u$ is either trivial or  generator of $H^*(X/G)$. Moreover, 
	\begin{enumerate}
		\item[(i)] If $u$ is a generator, then $X\sim_R \S^{4n+3}$, and 
		\item[(ii)] If $u$ is trivial, then $X\sim_R \S^3\times \mathbb{HP}^n. $		
	\end{enumerate}	
\end{theorem}
\begin{proof}
As $G$ is a compact Lie group which acts freely
on $X$, we have the Gysin sequence of the  sphere bundle $G\hookrightarrow X\stackrel{\pi}{\longrightarrow} X/G$:
\begin{align*}
\cdots\longrightarrow H^{i}(X)\stackrel{\rho}{\longrightarrow}H^{i-3}(X/G)\stackrel{\cup}{\longrightarrow}H^{i+1}(X/G) \stackrel{\pi^*}{\longrightarrow}H^{i+1}(X)\stackrel{\rho}{\longrightarrow}H^{i-2}(X/G)\longrightarrow\cdots
\end{align*}
which begins with
\begin{align*}
0\longrightarrow &H^3(X/G)\stackrel{\pi*}{\longrightarrow}H^3(X)\stackrel{\rho}{\longrightarrow}H^0(X/G)\stackrel{\cup}{\longrightarrow}H^4(X/G)\stackrel{\pi^*}{\longrightarrow}H^4(X)\longrightarrow\cdots
\end{align*}
Since $X/G\sim_R \mathbb{HP}^n$, we have $H^*(X/G)=R[a]/\langle a^{n+1}\rangle$, where $\deg a=4$.	
Note that $H^i(X)\cong H^i(X/G)$ for $i=0,1,2$.  By the exactness of the Gysin sequence,  $H^{4i+1}(X)=H^{4i+2}(X)=0$ for all $i\geq 0$ and $H^j(X)=0$ for all $j>4n+3$. There are three possibilities: If the  Euler class is (i)  generator, (ii) nontrivial but not a generator, and (iii)  trivial.\\

If the Euler class $u\in H^4(X/G)$ is a generator then   $\cup:H^{4i}(X/G)\to H^{4i+4}(X/G)$ is an isomorphism for all $0\leq i< n$ and thus,   the Euler class of the bundle $G\to X\stackrel{\pi}{\to }X/G$ is nonzero. By the exactness of the Gysin sequence $\rho:H^{4i+3}(X)\to H^{4i}(X/G)$ and $\pi^*:H^{4i+4}(X/G)\to H^{4i+4}(X)$ becomes trivial for all $0\leq i< n$. This gives that $H^{4i+3}(X)=H^{4i+4}(X)=0$ for all $0\leq i< n.$  As $H^{4n+4}(X/G)=0$, we have $H^{4n+3}(X)\cong  H^{4n}(X/G)\cong R$. 
Consequently, \begin{align*}
H^i(X)=\begin{cases}
R &\text{ if } i=0,4n+3\\
0 &\text{ otherwise. }
\end{cases}
\end{align*}
It is clear that $X\sim_R\S^{4n+3}$.\\

If $u\in H^4(X/G)$  is a nontrivial  but not a generator then
this  is possible only when $R=\Z$ and the Euler class $u\in  H^4(X/G)$ is $m.a$, where $m$ is an integer different from 0 and 1. Then, the Euler class   of the associated bundle  is $m.a$ and $\cup:H^{4i}(X/G)\to H^{4i+4}(X/G)$ maps generator $a^i$ to  $m.a^{i+1}$ for all $0\leq i<n$. By the exactness of the Gysin sequence, $H^{4i+3}(X)=0$ and $H^{4i+4}(X)\cong  H^{4i+4}(X/G)/\ker \pi^*\cong \Z_m$  for all $0\leq i< n$. As $H^{4n+4}(X/G)=0$, we have $H^{4n+3}(X)\cong H^{4n}(X/G)\cong\Z$.  Let $a_4\in H^4(X)$ and $b_{4n+3}\in H^{4n+3}(X) $ be such that $\pi^*(a)=a_4$ and $\rho(b_{4n+3})=a^n$.    Thus, we have 
\begin{align*}
H^{i}(X)=\begin{cases}
\Z &\text{ if } i=0 \text{ or }4n+3\\
\Z_m &\text{ if }0< i\equiv 0\; \text{(mod 4)} \leq  4n\\
0 &\text{ otherwise.}
\end{cases}
\end{align*}
As $G$ acts freely on $X$ and $B_G$ is simply connected, the $E_2$-term of the associated  Leray-Serre spectral sequence  for the Borel fibration $X\hookrightarrow X_G\to B_G$ is given by $E_2^{p,q}=H^p(B_G)\otimes H^q(X)$ which converges to $H^*(X_G)$ as an algebra.
Now, $H^*(B_G)=H^*(\mathbb{HP}^{\infty})=\Z[t],$ where $ \deg t=4$. Note that the only possible nontrivial differentials are $d_{4r}:E_{4r}^{*,*}\to E_{4r}^{*,*},1\leq r\leq n+1$. As $4n+4\geq 8$, $t\otimes 1$ and $1\otimes a_4$ are permanent cocycles. So, $H^4(X_G)\cong \Z\oplus \Z_m$,  a contradiction.

If the Euler class $u\in  H^4(X/G)$ is trivial then  the Euler class  of the bundle $G\to X\to X/G$ is zero and $\cup:H^{4i}(X/G)\to H^{4i+4}(X/G)$ is trivial for all $i\geq 0$. By the exactness of the Gysin sequence, $\rho:H^{4i+3}(X)\to H^{4i}(X/G)$ and $\pi^*:H^{4i}(X/G)\to H^{4i}(X)$ becomes isomorphism for all $0\leq i\leq n$. Let $a_4\in H^4(X)$ and $b_{4i+3}\in H^{4i+3}(X) $ be such that $\pi^*(a)=a_4$ and $\rho(b_{4i+3})=a^i$ for all $0\leq i\leq n$. This implies that $H^{4i+3}(X)\cong R$ with basis $\{b_{4i+3}\}$ and   $H^{4i}(X)\cong R$ with basis $\{a_4^{i}\}$ for all $0\leq i\leq n$. Thus, we have 
\begin{align*}
H^{i}(X)=\begin{cases}
R &\text{ if }0\leq i\equiv 0 \text{ or }3 \;(\text{mod } 4)\leq 4n+3\\
0 &\text{ otherwise.}
\end{cases}
\end{align*}
Note that $b_ib_j=0$ for all $i$ and $j$ and $a_4^{n+1}=0$. Next, we observe that $a_4^ib_3=b_{4i+3} $ for all $1\leq i\leq n$. In the associated  Leray-Serre spectral sequence,  the only possible nontrivial differentials are $d_{4r}:E_{4r}^{*,*}\to E_{4r}^{*,*}$, for $0\leq r\leq n+1$. So, the first nonzero possible differential is $d_4$. Clearly, $d_4(1\otimes a_4^i)=0$ for all $i\geq 0$. Now, we consider two subcases for coefficient groups $R=\Z_2$ or $R=\Z$:

Let $R=\Z_2$ and  $a_4^kb_3= 0$ for some $1\leq k\leq n$. 
If $d_4(1\otimes b_3)=t\otimes 1$, then $t\otimes a_4^k=d_4((1\otimes  a_4^k)(1\otimes  b_3))=0$ which is not possible. Therefore, $d_4(1\otimes b_3)=0$. As $d_{4r}:E_{4r}^{4i-4r,4r+2}\to E_{4r}^{4i,3}$ is trivial, $t^i\otimes b_3$  are permanent cocycles for all $i\geq 0$,  a contradiction to the fact that $H^j(X/G)=0$ for all $j>4n$. Therefore, $a^i_4b_3\neq 0$ for all $1\leq i\leq n$. This implies that $b_{4i+3}=a^i_4b_3$ for all $1\leq i\leq n$.
Thus, the  cohomology ring of $X$ is  $\Z_2[a_4,b_3]/\langle a_4^{n+1}, b_3^2 \rangle, \deg a_4=4, \deg b_3=3$. It is clear that $X\sim_{\Z_2}\mathbb{S}^3\times \mathbb{HP}^n$. This realizes case(ii)  of the theorem.

Now, let $R=\Z$ and $a_4^jb_3\not= \pm b_{4j+3}$ for some $1\leq j\leq n$.
Let $i_{0}\in \Z$ be the largest integer such that $a^{i_0}_4b_3\not=\pm b_{4i_{i_0}+3}$. If $d_4(1\otimes b_3)=0$, then  $\{t^i\otimes b_3\}$ are permanent cocycles for all $i\geq 0$, which is not possible as 
in subcase(i).  So, let $d_4(1\otimes b_{4i+3})=m_{i}(t\otimes a_4^i)$, where $m_i\in \Z$ and $m_0\not=0$. Then,   $H^4(X_G)\cong \Z\oplus\Z_{m_0}$. This gives that $m_0=\pm 1$. Clearly, $d_4:E_4^{0,4j+3}\to E_4^{4,4j}$ is an isomorphism for  $i_0+1\leq j\leq n$. So, we have $E_5^{i,4j}=E_5^{i,4j+3}=0$ for all $i\geq 0$, $j=0$ and $i_0+1\leq j\leq n$. Note that $E_5^{4i,4j}=\Z_{m_j}$, where $1\leq j\leq i_0$, and $E_5^{4i,4j+3}$ is $\Z$ if $m_j=0$, and trivial, otherwise. If $d_4:E_4^{0,4i_0+3}\to  E_4^{4,4i_0}$ is trivial, then $\{t^i\otimes b_{4i_0+3}\}_{i\geq 0}$ are permanent cocycles, a contradiction. So, let  $d_4:E_4^{0,4i_0+3}\to  E_4^{4,4i_0}$ is nontrivial.  Now, $d_4(1\otimes(a_4^{i_0}b_3\pm b_{4{i_0}+3}))=(m_0\pm m_{i_0})(t\otimes a_4^{i_0})$. Consequently, $m_{i_0}\not =\pm 1$. Thus,  $H^j(X_G)$ is nonzero for infinitely many values of $j$,  a contradiction. Therefore, $a_4^jb_3$ is $b_{4j+3}$ or $-b_{4j+3}$ for all $j$. Hence, $X\sim_{\Z}\mathbb{S}^3\times\mathbb{HP}^n$.  
\end{proof}

Now, we compute the orbit space of free actions of $G=\mathbb{S}^3$ on a paracompact space with integral cohomology $n$-sphere: 
\begin{theorem}\label{Theorem 2}
	Let $G=\S^3$ acts freely on a paracompact space $X$ with $X\sim_\Z\S^{n}$. Then, $n=4k+3$, for some $k\geq 0$ and $X/G\sim_\Z \mathbb{HP}^k$.
\end{theorem}
\begin{proof}
	  By the Gysin sequence sequence of the  3-sphere bundle, we get $H^0(X/G)\cong\Z$ and $H^i(X/G)=0$, for all $1\leq i\leq 3$  when $n\not=1,2 $ or 3.  Then, for $0\leq i\leq n-4$, $\cup:H^i(X/G)\to H^{i+4}(X/G)$ is an isomorphism. This gives that $H^i(X/G)=0$ for $0<i\equiv j\;(\text{mod } 4)<n,$ where $ 1\leq j\leq 3$ and $H^i(X/G)\cong\Z$  for $0\leq i\equiv 0(\text{mod }4)<n$ with basis $\{a^{\frac{i}{4}}\}$, where  $a\in H^4(X/G)$ denotes its generator. Suppose $n\equiv j\;(\text{mod }4)$, for some $0\leq j\leq 2$ then $H^{n-3}(X/G)=0$.  If ($n=1$ or 2) or $(0\leq j\leq 2)$, then by the exactness  of the Gysin sequence, $H^n(X/G)\not=0$,  which contradicts  Proposition \ref{H^j(X/G)=0}. Therefore, $n\equiv 3\; ( \text{mod } 4)$. Let $n=4k+3$ for some $k\geq 0$. For $n=3$, the result is trivially true. So let $n>3$. 
  Again, by Proposition \ref{H^j(X/G)=0},  $H^j(X/G)=0$ for all $j>n$, and hence  $a^{k+1}=0$. This implies that $\rho:H^n(X)\to H^{n-3}(X/G)$ is an isomorphism. Consequently,  $H^n(X/G)=0$.  Thus, we have, $H^*(X/G)=\Z[a]/\langle a^{k+1}\rangle, \deg a=4.$\qedhere
\end{proof}

	In 1963,  Su \cite{Su1963} has shown that if $G=\mathbb{S}^1$ acts freely on a space $X$ with orbit space $X/G\sim_\Z \mathbb{CP}^n$ and $\pi^*:H^2(X/G)\to H^2(X)$ is trivial,  then $X\sim_{\Z}\mathbb{S}^{2n+1}$, where $\pi:X\to X/G$ is the orbit map. In the next theorem, we discuss the case when $\pi^*$ is nontrivial.  
	\begin{theorem}\label{Theorem 3}
		Let $G=\mathbb{S}^1$ acts freely on a finitistic space $X$ with $X/G\sim_{\Z}\mathbb{CP}^n$, and $u\in H^2(X/G)$ be the Euler class of the bundle $G\to X\stackrel{\pi}{\to}X/G$. If the induced map $\pi^*:H^2(X/G)\to H^2(X)$ is nontrivial, then $u$ is trivial and $X\sim_{\Z} \mathbb{S}^1\times \mathbb{CP}^n$. 
	\end{theorem}
		\begin{proof}
	 As $X/G\sim_\Z \mathbb{CP}^n$, $H^*(X/G)=\Z[a]/\langle a^{\frac{n+1}{2}}\rangle $, where $\deg a=2$. As $\pi_1(B_G)=1$,  $E_2$-term of the Leray-Serre spectral sequence is $E_2^{p,q}=H^p(B_G)\otimes H^q(X)$ for the Borel fibration $X\hookrightarrow X_G\to B_G$. Note that the possible nontrivial differentials are $d_2,d_4,\cdots d_{2n+2}$. Suppose $\pi^*:H^2(X/G)\to H^2(X)$ is nontrivial. Then the Euler class $u\in H^2(X/G)$ is not a generator. So, first suppose that the Euler class of the principal bundle $X\stackrel{\pi}{\to}X/G$ is $m.a$, where $m\not =0$ in $\Z$. As $\pi^*:H^2(X/G)\to H^2(X)$ is nontrivial, $m\not =\pm 1$. Then by the exactness of the Gysin sequence $H^i(X)\cong \Z$ for $i=0,2n+1$; $H^i(X)\cong \Z_m$ with basis $\{a_2^{\frac{i}{2}}\}$ for $i=0,2,4,\cdots,2n $; and trivial otherwise. It gives that $t^i\otimes a_2^j$ are permanent cocycles for all $i,j\geq 0$, a contradiction. Next, suppose that the Euler class $u$ of the principal bundle is zero.   Consequently, we have 
	$$H^j(X)=\begin{cases}
	\Z &\text{ if } 0\leq j\leq 2n+1\\
	0 &\text{ otherwise. }
	\end{cases}$$
	 Let $a_2\in H^2(X)$ and $b_{2i+1}\in H^{2i+1}(X) $ be such that $\pi^*(a)=a_2$ and $\rho(b_{2i+1})=a^i$ for all $0\leq i\leq n$. This implies that $H^{2i+1}(X)\cong \Z$ with basis $\{b_{2i+1}\}$  and $H^{2i}(X)\cong \Z$ with basis $\{a_2^{i}\}$ for all $0\leq i\leq n$. Let if possible $a_2^jb_1\not= \pm b_{2j+1}$ for some $1\leq j\leq n$ and suppose $i_0$ be such an largest integer. As $H^1(X_G)=0$, $d_2(1\otimes b_1)\not=0$. So, let $d_2(1\otimes b_{2i+1})=m_{i}(t\otimes a_2^i)$, where $m_i\in \Z$ and $m_0\not=0$. 
	 Note that $E_3^{2i,2j}=\Z_{m_j}$ and $E_3^{2i,2j+1}$ is $\Z$ if $m_j=0$ and trivial otherwise for all $i\geq 0$ and $0\leq j\leq n$. Since $H^2(X_G)\cong \Z$, we have $d_2:E_2^{0,1}\to E_2^{2,0}$ is an isomorphism. Therefore, $E_3^{i,2j}=E_3^{i,2j+1}=0$ for all $i\geq 0$ and $i_0+1\leq j\leq n$.  If $d_2:E_2^{0,2i_0+1}\to  E_2^{2,2i_0}$ is trivial, then $\{t^i\otimes b_{2i_0+1}\}_{i\geq 0}$ are permanent cocycles, a contradiction. So, let  $d_2:E_2^{0,2i_0+1}\to  E_2^{2,2i_0}$ is nontrivial.
	 As $d_{2}(1\otimes a_2)=0$, we get $m_{i_0}\not=m_0$, and hence  $t^i\otimes a_2^{i_0}$ are permanent cocycles for all $i\geq 0$, a contradiction. Thus, $$H^*(X)=\Z[a_2,b_1]/\langle a_2^n,b_1^2\rangle,$$ where $\deg b_1=1 \text{ and }\deg a_2=2$. Hence, our claim.\end{proof}

Now, we prove  similar results with coefficients in $\mathbb{Z}_p$, $p$ a prime.

\begin{theorem}
	Let $G=\mathbb{S}^1$ acts freely on a finitistic space $X$ with the orbit space $X/G\sim_{\Z_p}\mathbb{CP}^n$, $p$ a prime. Let $\pi^*:H^2(X/G)\to H^2(X)$ be the map induced by the orbit map $\pi:X\to X/G$.
	\begin{enumerate}
		\item If  $\pi^*:H^2(X/G)\to H^2(X)$ is trivial, then $X\sim_{\Z_p}\mathbb{S}^{2n+1}$.
		\item If $\pi^*:H^2(X/G)\to H^2(X)$ is nontrivial, then either $X\sim_{\Z_p} \mathbb{S}^1\times \mathbb{CP}^n$ or $L_p^{2n+1}$.
	\end{enumerate}
	\begin{proof}
The Euler class of the principal bundle $X\to X/G$ is either trivial or a generator of $H^4(X/G;\mathbb{Z}_p)$. If the Euler class of the associated bundle is trivial, then $X\sim_{\Z_p} \mathbb{S}^{2n+1}.$ So, let the Euler class be a generator of $H^4(X;\mathbb{Z}_p)$. It is easy to see that
$$H^*(X;\mathbb{Z}_p)\cong \mathbb{Z}_p[b_1,b_2,\cdots,b_{2n+1},a_2]/\langle a_2^{n+1} \rangle, \deg a_2=2,\deg b_i=i.$$ In the Leray-Serre spectral sequence, we must have  $d_2(1\otimes b_1)\not= 0$ for suitable choice of generator $b_1$ and  $d_2(1\otimes a_2^i)=0$ for all $0\leq i\leq n$. This implies that $b_{2i+1}=a_2^ib_1$ for all $0\leq i\leq n$. If $b_1^2=0$, then  $X\sim_{\Z_p} \mathbb{S}^1\times \mathbb{CP}^n$. If $b_1^2\not= 0$ and $p=2$, then $a_2=b_1^2$. This gives that $X\sim_{\Z_2} \mathbb{RP}^{2n+1}.$  If $b^2_1\not= 0$ and $p$ is an odd prime, then $\beta (b_1)=a_2$, where $\beta:H^1(X; \mathbb{Z}_p)\to  H^2(X; \mathbb{Z}_p)$ is the Bockstein homomorphism associated to the coefficient sequence $0\to \mathbb{Z}_p \to \mathbb{Z}_{p^2}\to \mathbb{Z}_p\to 0$, then $X\sim_{\Z_p} L_p^{2n+1}$. \end{proof}
\end{theorem}
The next example realises the  above theorem.
\begin{example}
	Recall that the map $(\lambda,(z_0,z_1,\cdots,z_n))\to (\lambda z_0,\lambda z_1,\cdots,\lambda z_n)$, where $\lambda \in \mathbb{S}^1$ and $z_i\in \mathbb{C}$, $0\leq i\leq n$, defines a standard free action of $G=\mathbb{S}^1$ on $\mathbb{S}^{2n+1}$. The orbit space $X/G$ under this action is $\mathbb{CP}^n$. For $p$ a prime, $H=\langle e^{2\pi i/p}\rangle$ induces a free action on $\mathbb{S}^{2n+1}$ with the orbit space $\mathbb{S}^{2n+1}/H=L_p^{2n+1}$.  Consequently,  $\mathbb{S}^1=G/H$ acts freely on $L_p^{2n+1}$ with the orbit space $\mathbb{CP}^n$. Recall that for $p=2$, $L_p^{2n+1}=RP^{2n+1}$.
\end{example}

\bibliographystyle{plain}

\end{document}